\documentclass[11pt]{amsart}

\title[New free boundary minimal annuli of revolution]{New free boundary minimal annuli of revolution in the 3-sphere}
\author{Manuel Ruivo de Oliveira}
\address{Department of Mathematics, The University of British Columbia, 1984 Mathematics Road, Vancouver, V6T 1Z2, BC, Canada}
\email{m.oliveira@math.ubc.ca}

\usepackage{amsmath}
\usepackage{amsfonts}
\usepackage{amsthm}
\usepackage[shortlabels]{enumitem}
\usepackage{graphicx}
\usepackage[unicode=true,linktocpage,linkbordercolor={0.5 0.5 1},citebordercolor={0.5 1 0.5},linkcolor=blue]{hyperref}

\newcommand{\bb}[1]{\mathbb{#1}}
\DeclareMathOperator{\II}{II}

\newtheorem{theorem}{Theorem}[section]
\newtheorem{lemma}[theorem]{Lemma}
\newtheorem{proposition}[theorem]{Proposition}

\theoremstyle{definition}
\newtheorem{remark}[theorem]{Remark}
\theoremstyle{plain}

\thanks{This work was supported by the NSERC Discovery Grant 22R81123.}


\begin{document}
	
\begin{abstract}
	We rigorously establish the existence of many free boundary minimal annuli with boundary in a geodesic sphere of $\bb{S}^3$. These arise as compact subdomains of a one-parameter family of complete minimal immersions of $\bb{R} \times \bb{S}^1$ into $\bb{S}^3$ described by do Carmo and Dajczer \cite{doCarmo_dajczer_1983}. While the immersed free boundary minimal annuli we exhibit may in general fail to be embedded or contained in a geodesic ball, we show that there is at least a one-parameter family of embedded examples that are contained in geodesic balls whose radius may be less than, equal to or greater than $\frac{\pi}{2}$. After explaining the connection to Otsuki tori \cite{otsuki_1970}, we establish lower bounds on the number of immersed free boundary minimal annuli contained in each Otsuki torus in terms of the corresponding rational number. Finally, we show that some of the recent work of Lee and Seo \cite{lee_seo_2023} on isoperimetric inequalities and of Lima and Menezes \cite{lima_menezes_2023} on index bounds extends to geodesic balls equal to or larger than a hemisphere.
\end{abstract}

\maketitle
	
\section{Introduction}\label{sec:introduction}

A fundamental step in the study of any class of mathematical objects is the search for an initial collection of examples. It is from this initial collection that we may establish new conjectures as well as look for counterexamples to any future conjectures.

While there are now many examples of free boundary minimal surfaces in the Euclidean ball \cite{martin_li_2020}, some more and some less explicit, the situation is starkly different in geodesic balls of curved ambient spaces such as the sphere $\bb{S}^n$ or hyperbolic space $\bb{H}^n$. Here there are results of Fraser and Schoen \cite{fraser_schoen_2015} on uniqueness, of Li and Xiong \cite{li_xiong_2018} on a gap theorem, of Freidin and McGrath \cite{freidin_mcgrath_2019, freidin_mcgrath_2020} on area bounds, of Lee and Seo \cite{lee_seo_2023} on isoperimetric inequalities, and more. However, apart from the totally geodesic free boundary disks, the only other examples of free boundary minimal surfaces in geodesic balls of $\bb{S}^n$ and $\bb{H}^n$ appearing in the literature are a collection of rotational annuli (see \cite{li_xiong_2018}) thought to be free boundary but for which a rigorous proof is still missing.

The purpose of this paper is to remedy this situation by providing an abundance of parametrized examples in the sphere $\bb{S}^3$, all of them immersed free boundary minimal annuli of revolution. Our point of departure is the one-parameter family of minimal immersions of $\bb{R} \times \bb{S}^1$ into $\bb{S}^3$ described by do Carmo and Dajczer \cite{doCarmo_dajczer_1983}, which includes the countable collection of minimal tori known as Otsuki tori \cite{otsuki_1970}. For the purposes of existence, we find it useful to consider not only free boundary surfaces contained in a geodesic ball, but more generally free boundary surfaces with boundary in a geodesic sphere, whether they are contained in a ball or not. Indeed, we find many immersed free boundary minimal annuli with boundary in a geodesic sphere that are not contained in a geodesic ball (see, for example, the second plot in Figure \ref{fig:otsuki_2_3_phi0_0_fbma}). Of those that are contained in a ball, we show that some are contained in a ball smaller than a hemisphere, some in the hemisphere, and some in a ball larger than a hemisphere. Many have self-intersections, but we guarantee that at least some are embedded.

We note that none of the works cited above deal with free boundary minimal surfaces in geodesic balls of $\bb{S}^3$ larger than a hemisphere, which naturally raises the question of whether the theory developed therein still applies. In this direction, we show that the recent results of Lee and Seo \cite{lee_seo_2023} on isoperimetric inequalities and of Lima and Menezes \cite{lima_menezes_2023} on index bounds can be extended to this new setting.

\section{Complete minimal surfaces of revolution in $\bb{S}^3$}\label{sec:complete_minimal_surfaces}

We begin by describing a one-parameter family of minimal immersions of $\bb{R} \times \bb{S}^1$ into $\bb{S}^3$, due to do Carmo and Dajczer \cite{doCarmo_dajczer_1983}. The proof is a simplified version of an argument of Mori \cite{mori_1981}.

\begin{proposition}[\cite{doCarmo_dajczer_1983}] \label{prop:complete_minimal_immersions}
	For $a \in (-\frac{1}{2}, \frac{1}{2})$ and $\phi_0, t, s \in \bb{R}$, let
	\begin{align}
		\psi(a, t) &= \frac{(\frac{1}{4} - a^2)^\frac{1}{2}}{(\frac{1}{2} + a\cos 2t)^\frac{1}{2} (\frac{1}{2} - a\cos 2t)}, \label{eq:psi_s3}\\
		\phi(a, s) &= \phi_0 + \int_0^s \psi(a, t) \, dt,
	\end{align}
	and define the curve $\gamma_a(s) = (x_a(s), y_a(s), z_a(s))$ by
	\begin{align}
		x_a(s) &= \left(\frac{1}{2} - a\cos 2s\right)^\frac{1}{2} \cos\phi(a, s),\\
		y_a(s) &= \left(\frac{1}{2} - a\cos 2s\right)^\frac{1}{2} \sin\phi(a, s), \label{eq:second_coordinate_generating_curve_s3}\\
		z_a(s) &= \left(\frac{1}{2} + a\cos 2s\right)^\frac{1}{2}. \label{eq:third_coordinate_generating_curve_s3}
	\end{align}
	Then $X_a: \bb{R} \times \bb{S}^1 \to \bb{S}^3$ given by
	\begin{equation}
		X_a(s, \theta) = (x_a(s), y_a(s), z_a(s)\cos\theta, z_a(s)\sin\theta)
	\end{equation}
	is a complete immersed minimal surface of revolution in $\bb{S}^3$.
\end{proposition}

\begin{proof}
	Let $\gamma(s) = (x(s), y(s), z(s))$ be a smooth curve in $\bb{S}^2 \cap \{x^3 > 0\}$, parametrized by arc length, and rotate it about the $x^1 x^2$-plane to get $X: \bb{R} \times \bb{S}^1 \to \bb{S}^3$ given by $X(s, \theta) = (x(s), y(s), z(s)\cos\theta, z(s)\sin\theta)$. Then $\Sigma$, the image of $X$, is minimal in $\bb{S}^3$ if and only if $\Delta_\Sigma x^i + 2x^i = 0$ for $i=1,\dots,4$, which results in five equations on the generating curve $\gamma$:
	\begin{align}
		x^2 + y^2 + z^2 & = 1, \label{eq:gamma_in_s2}\\
		\dot{x}^2 + \dot{y}^2 + \dot{z}^2 & = 1, \label{eq:gamma_parametrized_by_arc_length_s3}\\
		\ddot{x} + \frac{\dot{z}}{z} \dot{x} + 2x & = 0, \label{eq:first_component_gamma_ode_s3}\\
		\ddot{y} + \frac{\dot{z}}{z} \dot{y} + 2y & = 0, \label{eq:second_component_gamma_ode_s3}\\
		\ddot{z} + \frac{\dot{z}^2 - 1}{z} + 2z & = 0. \label{eq:third_component_gamma_ode_s3}
	\end{align}
	Solving \eqref{eq:third_component_gamma_ode_s3} gives $z(s) = (\frac{1}{2} + a\cos 2s)^\frac{1}{2}$ for $a \in (-\frac{1}{2}, \frac{1}{2})$, and letting
	\begin{align}
		x(s) &= (1 - z^2(s))^\frac{1}{2} \cos \phi(s), \label{eq:first_component_gamma_expression_s3}\\
		y(s) &= (1 - z^2(s))^\frac{1}{2} \sin \phi(s), \label{eq:second_component_gamma_expression_s3}
	\end{align}
	implies
	\begin{equation}
		\phi(s) = \phi_0 + \int_{0}^{s} \frac{(\frac{1}{4} - a^2)^\frac{1}{2}}{(\frac{1}{2} + a\cos 2t)^\frac{1}{2}(\frac{1}{2} - a\cos 2t)} \, dt,
	\end{equation}
	for some $\phi_0 \in \bb{R}$. We can now check directly that the differential equations \eqref{eq:first_component_gamma_ode_s3} and \eqref{eq:second_component_gamma_ode_s3} are satisfied, so that $X$ is indeed a minimal immersion.
\end{proof}

\begin{remark} \label{rem:clifford_torus}
	In the special case that $a = \phi_0 = 0$, we find, directly from the definition, that $\psi(0, t) = \sqrt{2}$ for all $t$, that $\phi(0, s) = \sqrt{2}s$ for all $s$, and therefore the generating curve $\gamma_0(s) = (\frac{1}{\sqrt{2}} \cos (\sqrt{2}s), \frac{1}{\sqrt{2}} \sin (\sqrt{2}s), \frac{1}{\sqrt{2}})$ parametrizes a circle in $\bb{S}^2 \cap \{x^3 > 0\}$. The immersion given by the previous proposition becomes $X_0(s, \theta) = (\frac{1}{\sqrt{2}} \cos (\sqrt{2}s), \frac{1}{\sqrt{2}} \sin (\sqrt{2}s), \frac{1}{\sqrt{2}} \cos\theta, \frac{1}{\sqrt{2}} \sin\theta)$, which we recognize as a parametrization of the Clifford torus $(x^1)^2 + (x^2)^2 = (x^3)^2 + (x^4)^2 = \frac{1}{2}$.
\end{remark}

The image $\Sigma_a(\phi_0)$ of the immersion given in Proposition \ref{prop:complete_minimal_immersions} is in general an immersed cylinder $\bb{R} \times \bb{S}^1$. However, just as in Remark \ref{rem:clifford_torus}, it may happen that the parametrization closes the factor of $\bb{R}$ in its domain and returns an immersed torus $\bb{S}^1 \times \bb{S}^1$. We now characterize the set of all $a \in (-\frac{1}{2}, \frac{1}{2}) \setminus \{0\}$ for which this is the case. Let $C_a: a \in (-\frac{1}{2}, \frac{1}{2}) \mapsto \int_{0}^{\pi} \psi(a, t) \, dt$ be the increase in $\phi(a, s)$ as $s$ goes from $0$ to $\pi$.

\begin{lemma} \label{lem:immersed_torus_rational_pi}
	For $a \in (-\frac{1}{2}, \frac{1}{2}) \setminus \{0\}$, the surface $\Sigma_a(\phi_0)$ is an immersed torus if and only if $C_a$ is a rational multiple of $\pi$.
\end{lemma}

\begin{proof}
	Let $a \in (-\frac{1}{2}, \frac{1}{2}) \setminus \{0\}$. First note that $\psi(a, \cdot)$ is a function of $\cos 2t$ and hence is $\pi$-periodic. Therefore $\phi(a, s + \pi) = \phi(a, s) + C_a$ for all $s \in \bb{R}$, and more generally $\phi(a, s + k\pi) = \phi(a, s) + kC_a$ for all $s \in \bb{R}$ and $k \in \bb{Z}$.

	As a surface of revolution, $\Sigma_a(\phi_0)$ is an immersed torus exactly when the generating curve $\gamma_a$ is periodic. In one direction, if $\gamma_a = (x_a, y_a, z_a)$ is periodic, then all its component functions are periodic and have a common period. But $z_a$ is $\pi$-periodic, and it is not constant since $a \neq 0$, so $x_a$ and $y_a$ must be $k\pi$-periodic for some $k \in \bb{N}$. Then $\cos \phi(a, s + k\pi) = \cos \phi(a, s)$ and $\sin \phi(a, s + k\pi) = \sin \phi(a, s)$ for all $s$, which gives $\phi(a, s + k\pi) - \phi(a, s) \in 2\pi\bb{Z}$, and, from the previous paragraph, $k C_a \in 2\pi\bb{Z}$, as required. The converse can simply be checked directly.
\end{proof}

The case $a = 0$ is excluded from the previous lemma since, as in Remark \ref{rem:clifford_torus}, $\Sigma_0(\phi_0)$ is an immersed torus, but $C_0 = \sqrt{2} \pi \notin \pi\bb{Q}$.

The minimal tori of Lemma \ref{lem:immersed_torus_rational_pi} were introduced by Otsuki \cite{otsuki_1970} through a different approach and have been studied by many authors \cite{doCarmo_dajczer_1983, hsiang_lawson_1971, hu_song_2012, otsuki_1970, otsuki_1993, penskoi_2013}. In the survey \cite{otsuki_1993}, Otsuki explains that these tori arise from periodic solutions $x(t)$ of the nonlinear differential equation
\begin{equation}
	2x(1 - x^2)\ddot{x} + \dot{x}^2 + (1 - x^2)(2x^2 - 1) = 0,
\end{equation}
with fundamental period given by
\begin{equation}
	T(c) = \sqrt{2c} \int_{x_0(c)}^{x_1(c)} \frac{dx}{x\sqrt{(2-x) \left(x(2-x) - c \right)}}
\end{equation}
for $c \in (0, 1)$, where $x_0(c) = 1 - \sqrt{1 - c}$ and $x_1(c) = 1 + \sqrt{1 - c}$ (see \cite[Equation (2.9)]{otsuki_1993}).

We want to use the results of \cite{otsuki_1970, otsuki_1993} about $T(c)$ to deduce properties of $C_a$. To do that, we first establish the connection between the two.

\begin{lemma} \label{lem:connection_c_a_t_c}
	$C_a = T(1 - 4a^2)$ for $a \in (0, \frac{1}{2})$.
\end{lemma}

\begin{proof}
	The proof is a variable substitution and some computation. From the definition, $\psi(a, \pi - t) = \psi(a, t)$ for all $a$ and $t$, that is, $\psi(a, \cdot)$ is symmetric about $t = \frac{\pi}{2}$. This allows us to write $C_a$ as
	\begin{equation} \label{eq:c_a_integral}
		C_a = 2 \int_{0}^{\frac{\pi}{2}} \frac{(\frac{1}{4} - a^2)^\frac{1}{2}}{(\frac{1}{2} + a\cos 2t)^\frac{1}{2} (\frac{1}{2} - a\cos 2t)} \, dt.
	\end{equation}
	Now let $x = 1 - 2a \cos 2t$. Then $\frac{1}{2} - a\cos 2t = \frac{x}{2}$, $\frac{1}{2} + a\cos 2t = 1 - \frac{x}{2}$, and
	\begin{equation}
		\begin{aligned}
			\frac{dx}{dt} &= -2a (-2\sin 2t) \\
			&= 2 (4a^2 - (1 - x)^2)^\frac{1}{2} \\
			&= 2 (x(2 - x) - (1 - 4a^2))^\frac{1}{2}.
		\end{aligned}
	\end{equation}
	For the bounds, when $t = 0$, $x = 1 - 2a$, and when $t = \frac{\pi}{2}$, $x = 1 + 2a$. Substituting all this back into \eqref{eq:c_a_integral} yields the result.
\end{proof}

\begin{lemma} \label{lem:properties_of_c_a}
	The following are true of $C_a: (-\frac{1}{2}, \frac{1}{2}) \to \bb{R}$:
	\begin{enumerate}[i.]
		\item $C_a$ is even as a function of $a$; \label{eq:c_a_even}
		\item $C_0 = \sqrt{2}\pi$; \label{eq:c_zero}
		\item $C_a$ is strictly decreasing in $[0, \frac{1}{2})$; \label{eq:c_a_strictly_decreasing}
		\item $C_a > \pi$ for all $a \in (-\frac{1}{2}, \frac{1}{2})$; \label{eq:c_a_lower_bound}
		\item $\lim_{a \to \frac{1}{2}^-} C_a = \pi$. \label{eq:c_a_right_limit}
	\end{enumerate}
\end{lemma}

\begin{proof}
	As in the beginning of the proof of Lemma \ref{lem:immersed_torus_rational_pi}, for each $a \in (-\frac{1}{2}, \frac{1}{2}) \setminus \{0\}$, $\psi(a, \cdot)$ is $\pi$-periodic. This, together with the glide symmetry $\psi(-a, t) = \psi(a, t + \frac{\pi}{2})$ for $a \in (-\frac{1}{2}, \frac{1}{2})$ and $t \in \bb{R}$, gives \eqref{eq:c_a_even}. Point \eqref{eq:c_zero} is immediate from the definition. 

	While \eqref{eq:c_a_lower_bound} follows from the other points combined, we include here a direct proof that does not rely on \cite{otsuki_1970, otsuki_1993}. The lower bound on $\psi$,
	\begin{equation}
		\psi(a, t) \ge \frac{(\frac{1}{2} - a)^\frac{1}{2}}{\frac{1}{2} - a\cos 2t},
	\end{equation}
	gives a lower bound on $C_a$,
	\begin{equation}
		\begin{aligned}
			C_a &\ge \int_{0}^{\pi} \frac{(\frac{1}{2} - a)^\frac{1}{2}}{\frac{1}{2} - a\cos 2t} \, dt\\
			&= 2 \int_{0}^{\frac{\pi}{2}} \frac{(\frac{1}{2} - a)^\frac{1}{2}}{\frac{1}{2} - a\cos 2t} \, dt\\
			&= 2 \lim_{s \to \frac{\pi}{2}^-} \frac{1}{(\frac{1}{2} + a)^\frac{1}{2}} \arctan \left(\frac{(\frac{1}{2} + a)^\frac{1}{2}}{(\frac{1}{2} - a)^\frac{1}{2}} \tan s \right)\\
			&= \frac{\pi}{(\frac{1}{2} + a)^\frac{1}{2}}.
		\end{aligned}
	\end{equation}
	In particular, $C_a > \pi$ for all $a \in (-\frac{1}{2}, \frac{1}{2})$, establishing \eqref{eq:c_a_lower_bound}.

	To show \eqref{eq:c_a_strictly_decreasing} and \eqref{eq:c_a_right_limit} we make use of Lemma \ref{lem:connection_c_a_t_c}. From \cite[Section 5]{otsuki_1993}, $T(c)$ is strictly increasing for $c \in (0, 1)$, so Lemma \ref{lem:connection_c_a_t_c} implies that $\frac{d}{da} C_a$ is negative for $a \in (0, \frac{1}{2})$, giving \eqref{eq:c_a_strictly_decreasing}. Finally, \eqref{eq:c_a_right_limit} follows from $\lim_{c \to 0^+} T(c) = \pi$ and Lemma \ref{lem:connection_c_a_t_c} (see \cite[Appendix]{otsuki_1970}).
\end{proof}

Points \eqref{eq:c_a_even}, \eqref{eq:c_zero} and \eqref{eq:c_a_strictly_decreasing} of the previous lemma give us the upper bound $C_a \le \sqrt{2}\pi$ for $a \in (-\frac{1}{2}, \frac{1}{2})$. Point \eqref{eq:c_a_lower_bound} gives a lower bound, \eqref{eq:c_a_strictly_decreasing} gives injectivity of $C_a$ in $[0, \frac{1}{2})$, and \eqref{eq:c_zero}, \eqref{eq:c_a_strictly_decreasing} and \eqref{eq:c_a_right_limit} imply surjectivity of $C_a$ onto $(\pi, \sqrt{2}\pi]$. Together, Lemma \ref{lem:properties_of_c_a} shows that, given $p,q \in \bb{N}$, $p,q$ coprime, with $\frac{1}{2} < \frac{p}{q} < \frac{1}{\sqrt{2}}$, there is a unique $a \in (0, \frac{1}{2})$ such that $C_a = \frac{2p\pi}{q}$. The resulting surface $\Sigma_a(\phi_0)$ is then the Otsuki torus corresponding to the rational number $\frac{p}{q}$, with initial angle $\phi_0$.

\section{Free boundary minimal annuli of revolution with boundary in a geodesic sphere of $\bb{S}^3$} \label{sec:fbma_general_a}

Now we look for free boundary minimal surfaces inside the one-parameter family of complete minimal surfaces described in the previous section. It turns out that many of the compact free boundary minimal surfaces we find will have their boundary in a geodesic sphere of $\bb{S}^3$, but will not be contained in a geodesic ball of $\bb{S}^3$. This is of course impossible in Euclidean space as can be seen from the convex hull property. It is therefore useful to consider all free boundary minimal surfaces with boundary in a geodesic sphere for the purposes of existence, and only then return to the question of whether they are contained in a ball.

Let $p_N = (1, 0, 0, 0) \in \bb{S}^3$ be the north pole and $B_r(p_N)$ be the closed geodesic ball in $\bb{S}^3$ centered at $p_N$ with radius $r$. The next lemma introduces a function $f_a$ whose zeros characterize orthogonal intersections of $X_a$ with some geodesic sphere $\partial B_r(p_N)$ and gives a sufficient condition for the existence of an immersed free boundary minimal annulus with boundary in such a geodesic sphere.

\begin{lemma} \label{lem:characterization_orthogonal_intersections}
	For $a \in (-\frac{1}{2}, \frac{1}{2})$, let $f_a: \bb{R} \to \bb{R}$ be defined by
	\begin{equation} \label{eq:definition_f_a_s3}
		f_a(s) = a\sin(2s) \sin\phi(a, s) + \left(\frac{1}{4} - a^2 \right)^\frac{1}{2} \left(\frac{1}{2} + a\cos 2s \right)^\frac{1}{2} \cos\phi(a, s).
	\end{equation}
	If $s_1 < s_2 \in \bb{R}$ are such that $f_a(s_1) = f_a(s_2) = 0$ and $x_a(s_1) = x_a(s_2)$, then $X_a: [s_1, s_2] \times \bb{S}^1 \to \bb{S}^3$ as given in Proposition \ref{prop:complete_minimal_immersions} is an immersed free boundary minimal annulus with boundary in the geodesic sphere $\bb{S}^3 \cap \{x^1 = x_a(s_1)\}$.
\end{lemma}

\begin{proof}
	The vector field $\nu$ on $\bb{S}^3 \setminus \{\pm p_N\}$ given by
	\begin{equation}
		\nu(p) = \frac{1}{(1-(x^1)^2)^\frac{1}{2}} \left( ((x^1)^2 - 1) \partial_1 + x^1 x^2 \partial_2 + x^1 x^3 \partial_3 + x^1 x^4 \partial_4 \right)
	\end{equation}
	restricts to the outward unit conormal to $B_r(p_N)$ along its boundary whenever $p \in \partial B_r(p_N)$. Hence, letting $\nu_a = \nu \circ X_a$, we have
	\begin{equation}
		\nu_a(s, \theta) = \frac{1}{(1 - x_a^2)^\frac{1}{2}} \left( (x_a^2 - 1)\partial_1 + x_a y_a \partial_2 + x_a z_a \cos(\theta) \partial_3 + x_a z_a \sin(\theta) \partial_4 \right).
	\end{equation}
	To compute the unit normal $N_a$ to $\Sigma_a(\phi_0)$, start with the unit vector field $n_a = \gamma_a \times \dot{\gamma}_a$ along the generating curve $\gamma_a$ that is normal to the curve but tangent to $\bb{S}^2$, and rotate it around the $x^1 x^2$-plane to get
	\begin{multline}
		N_a(s, \theta) = (y_a\dot{z}_a - z_a\dot{y}_a) \partial_1 - (x_a\dot{z}_a - z_a\dot{x}_a) \partial_2\\
		+ (x_a\dot{y}_a - y_a\dot{x}_a) \cos(\theta) \partial_3 + (x_a\dot{y}_a - y_a\dot{x}_a) \sin(\theta) \partial_4.
	\end{multline}
	It follows that
	\begin{equation}
		\langle N_a, \nu_a \rangle (s, \theta) = - \frac{y_a\dot{z}_a - z_a\dot{y}_a}{(1-x_a^2)^\frac{1}{2}},
	\end{equation}
	and, making use of \eqref{eq:second_coordinate_generating_curve_s3} and \eqref{eq:third_coordinate_generating_curve_s3}, a direct computation gives
	\begin{equation}
		y_a\dot{z}_a - z_a\dot{y}_a = -\frac{a\sin(2s) \sin \phi(a, s) + (\frac{1}{4} - a^2)^\frac{1}{2} (\frac{1}{2} + a\cos 2s)^\frac{1}{2} \cos\phi(a, s)}{(\frac{1}{2} + a\cos 2s)^\frac{1}{2} (\frac{1}{2} - a\cos 2s)^\frac{1}{2}}.
	\end{equation}
	This shows that the zeros of $f_a$ characterize orthogonal intersections of $\Sigma_a(\phi_0)$ with geodesic spheres in $\bb{S}^3$ centered at $p_N$. The hypothesis $x_a(s_1) = x_a(s_2)$ guarantees that the two boundary components $X_a(\{s_1\}\times\bb{S}^1)$ and $X_a(\{s_2\}\times\bb{S}^1)$ lie on the same geodesic sphere.
\end{proof}

Note that the previous lemma says nothing about the topology of the image of the resulting immersion. For example, when $a = 0$, then $s_1 = -\frac{\pi}{2\sqrt{2}}$ and $s_2 = \frac{3\pi}{2\sqrt{2}}$ satisfy the hypotheses of the lemma, but the image of the resulting immersion is a topological torus. In fact there are many examples like this whenever the generating curve $\gamma_a$ is periodic.

We now show that, for each $a \in (-\frac{1}{2}, \frac{1}{2})$, the complete surface $\Sigma_a(0)$ contains an infinite sequence of immersed free boundary minimal annuli $X_a: I_i \times \bb{S}^1 \to \bb{S}^3$, $i \in \bb{N}$, with $I_i \subset \bb{R}$ a compact interval, and moreover that this sequence is nested, that is, $X_a(I_i \times \bb{S}^1) \subset X_a(I_{i+1} \times \bb{S}^1)$ for $i \in \bb{N}$. See Figure \ref{fig:fbma_general_a} for the first four such immersed free boundary minimal annuli in $\Sigma_a(0)$ with $a = 0.29$.

\begin{theorem} \label{thm:infinitely_many_fbma}
	For each $a \in (-\frac{1}{2}, \frac{1}{2})$, the surface $\Sigma_a(0)$ contains a countably infinite nested collection of immersed free boundary minimal annuli, each with boundary in a geodesic sphere centered at $p_N$.
\end{theorem}

\begin{proof}
	Let $a \in (-\frac{1}{2}, \frac{1}{2})$ and note that $\psi(a, \cdot)$ is even. Letting $\phi_0 = 0$ guarantees that $\phi(a, \cdot)$ is odd and hence that $f_a$ and $x_a$ are even. This reduces the problem to finding positive zeros of $f_a$. For suppose $s_* > 0$ is such that $f_a(s_*) = 0$. Then the above implies that $f_a(-s_*) = 0$ and $x_a(-s_*) = x_a(s_*)$, so by Lemma \ref{lem:characterization_orthogonal_intersections}, $X_a: [-s_*, s_*] \times \bb{S}^1 \to \bb{S}^3$ is an immersed free boundary minimal annulus with boundary in the geodesic sphere $\bb{S}^3 \cap \{x^1 = x_a(s_*)\}$.
	
	We now prove the existence of countably many positive zeros of $f_a$. As shown in the proof of Lemma \ref{lem:immersed_torus_rational_pi}, $\phi(a, s + k\pi) = \phi(a, s) + kC_a$ for all $s \in \bb{R}$ and $k \in \bb{Z}$. Using $\phi_0 = 0$, we find that $\phi(a, k\pi) = kC_a$ for all $k \in \bb{Z}$. As in the beginning of the proof of Lemma \ref{lem:connection_c_a_t_c}, $\psi(a, \cdot)$ being symmetric about $t = \frac{\pi}{2}$ gives $C_a = 2 \int_{0}^{\frac{\pi}{2}} \psi(a, t) \, dt$, and then $\phi_0 = 0$ means that $\phi(a, \frac{\pi}{2}) = \frac{C_a}{2}$. By induction, $\phi(a, \frac{k\pi}{2}) = \frac{kC_a}{2}$ for all $k \in \bb{Z}$, and hence
	\begin{equation} \label{eq:f_half_integer_multiples_pi}
		f_a \left(\frac{k\pi}{2} \right) = \left(\frac{1}{4} - a^2 \right)^\frac{1}{2} \left(\frac{1}{2} + a (-1)^k \right)^\frac{1}{2} \cos \left(\frac{kC_a}{2} \right).
	\end{equation}
	Note that, for any $\beta \in (0, 2\pi)$, $\cos(k\beta)$ is positive for infinitely many $k \in \bb{N}$ and negative for infinitely many $k \in \bb{N}$. This holds since, if $\beta \in (0, 2\pi) \cap \pi\bb{Q}$, then its orbit $\{k\beta\}_{k\in\bb{N}}$ is periodic in $\bb{S}^1$, and if $\beta \in (0, 2\pi) \setminus \pi\bb{Q}$, then its orbit is dense in $\bb{S}^1$. But from Lemma \ref{lem:properties_of_c_a}, $\frac{C_a}{2} \in (\frac{\pi}{2}, \frac{\pi}{\sqrt{2}}]$, so from \eqref{eq:f_half_integer_multiples_pi}, $f_a(\frac{k\pi}{2})$ is positive for infinitely many $k \in \bb{N}$ and negative for infinitely many $k \in \bb{N}$. This shows $f_a$ has countably many positive zeros $s_1 < s_2 < \dots < s_i < \dots$, $i \in \bb{N}$, and then $X_a: [-s_i, s_i] \times \bb{S}^1 \to \bb{S}^3$ is an immersed free boundary minimal annulus with boundary in the geodesic sphere $\bb{S}^3 \cap \{x^1 = x_a(s_i)\}$. The collection is nested because $[-s_i, s_i] \subset [-s_{i+1}, s_{i+1}]$ for every $i \in \bb{N}$.
\end{proof}

\begin{figure}
	\centering
	\includegraphics[width=\textwidth]{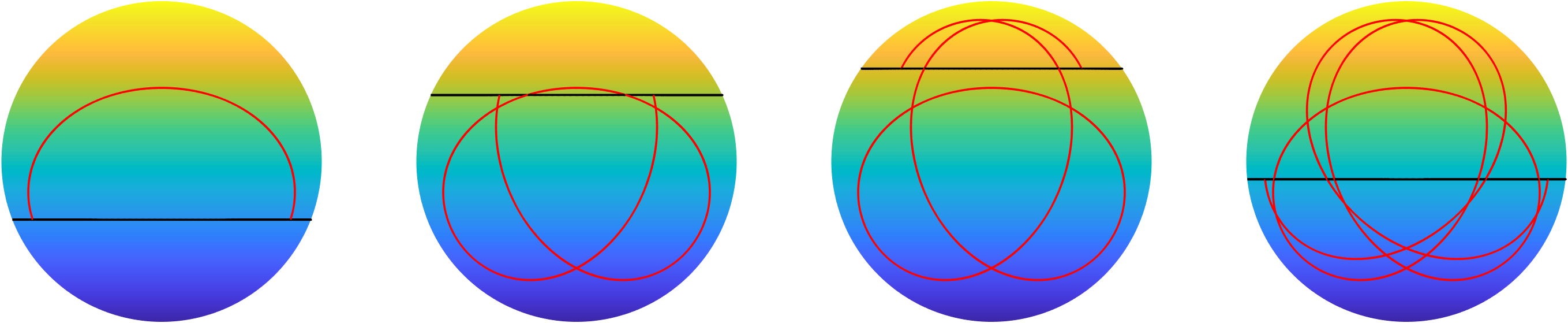}
	\caption{The sphere $\bb{S}^2$, the generating curve $\gamma_a$ (red) for the immersed free boundary minimal annuli $X_a: [-s_i, s_i] \times \bb{S}^1 \to \bb{S}^3$ ($i = 1, \dots, 4$) and the geodesic spheres $\partial B_{r}^2(p_N)$ (black) that $\gamma_a$ intersects orthogonally, where $a = 0.29$}
	\label{fig:fbma_general_a}
\end{figure}

While many of the free boundary minimal annuli we describe have self-intersections, there is at least a one-parameter family of such surfaces that are also embedded.

\begin{proposition} \label{prop:one_parameter_family_embedded}
	Let $\phi_0 = 0$. For each $a \in (-\frac{1}{2}, \frac{1}{2})$, let $s_1(a)$ be the first positive zero of $f_a$. Then the free boundary minimal annulus $X_a: [-s_1(a), s_1(a)] \times \bb{S}^1 \to \bb{S}^3$ with boundary in the geodesic sphere $\bb{S}^3 \cap \{x^1 = x_a(s_1(a))\}$ is embedded.
\end{proposition}

\begin{proof}
	We already know from Proposition \ref{prop:complete_minimal_immersions} that $X_a: [-s_1(a), s_1(a)] \times \bb{S}^1 \to \bb{S}^3$ is a smooth immersion. Since $[-s_1(a), s_1(a)] \times \bb{S}^1$ is compact, it only remains to show that $X_a: [-s_1(a), s_1(a)] \times \bb{S}^1 \to \bb{S}^3$ is injective.
	
	If $a = 0$, then we have Remark \ref{rem:clifford_torus} and it is easy to check directly that $s_1(0) = \frac{\pi}{2\sqrt{2}}$ and that the map is injective. So assume that $a \neq 0$.
	
	We show that $s_1(a) < \frac{\pi}{2}$. Computing from the definition, $f_a(0) = (\frac{1}{4} - a^2)^\frac{1}{2} (\frac{1}{2} + a)^\frac{1}{2}$, which is positive, and $f_a(\frac{\pi}{2}) = (\frac{1}{4} - a^2)^\frac{1}{2} (\frac{1}{2} - a)^\frac{1}{2} \cos(\frac{C_a}{2})$. But $\frac{C_a}{2} \in (\frac{\pi}{2}, \frac{\pi}{\sqrt{2}}]$, so $f_a(\frac{\pi}{2})$ is negative and $s_1(a) < \frac{\pi}{2}$.
	
	This allows us to show that $y_a(s)$ has the same sign as $s$ for $s \in [-s_1(a), s_1(a)]$. From the definition, $\dot{\phi} = \psi > 0$, so $\phi(a, \cdot)$ is strictly increasing, $\phi(a, 0) = 0$ and, from the proof of Theorem \ref{thm:infinitely_many_fbma}, $\phi(a, \frac{\pi}{2}) = \frac{C_a}{2} \in (\frac{\pi}{2}, \frac{\pi}{\sqrt{2}}]$. Therefore $\phi(a, s) \in (0, \frac{\pi}{\sqrt{2}})$ for $s \in (0, s_1(a)]$ and hence $y_a(s) > 0$ in the same range. Since $y_a$ is odd, the claim follows.
	
	The injectivity of $X_a$ is equivalent to the injectivity of the generating curve $\gamma_a$. Suppose $\gamma_a(s_*) = \gamma_a(\tilde{s})$ for some $s_*, \tilde{s} \in [-s_1(a), s_1(a)]$. Then, in particular, $y_a(s_*) = y_a(\tilde{s})$ and $z_a(s_*) = z_a(\tilde{s})$. From $z_a(s_*) = z_a(\tilde{s})$ we find $\cos 2s_* = \cos 2\tilde{s}$, and since $s_*, \tilde{s} \in (-\frac{\pi}{2}, \frac{\pi}{2})$, either $s_* = \tilde{s}$ or $s_* = -\tilde{s}$. In the first case we are done, so suppose $s_* = -\tilde{s}$. Then $y_a(s_*) = y_a(\tilde{s})$ and the fact that $y_a$ is odd gives $y_a(s_*) = 0$, and the previous paragraph implies that $s_* = 0$, so once again we have $s_* = \tilde{s}$, as required.
\end{proof}

Now we show that each embedded free boundary minimal annulus in the one-parameter family described in the previous proposition is in fact contained in a geodesic ball, and the sign of the parameter determines whether that ball is smaller than, equal to or greater than a hemisphere. See Figure \ref{fig:fbma_positive_negative_a}.

\begin{proposition} \label{prop:one_parameter_family_in_ball}
	Let $\phi_0 = 0$. For each $a \in (-\frac{1}{2}, \frac{1}{2})$, let $s_1(a)$ be the first positive zero of $f_a$. Then the embedded free boundary minimal annulus $X_a: [-s_1(a), s_1(a)] \times \bb{S}^1 \to \bb{S}^3$ is contained in the geodesic ball $B_{r(a)}(p_N) = \bb{S}^3 \cap \{x^1 \ge x_a(s_1(a))\}$, and
	\begin{equation} \label{eq:radius_geodesic_balls}
		\left\{
		\begin{array}{ll}
			r(a) < \frac{\pi}{2}, & \text{if } a \in (-\frac{1}{2}, 0),\\ [5pt]
			r(a) = \frac{\pi}{2}, & \text{if } a = 0,\\ [5pt]
			r(a) > \frac{\pi}{2}, & \text{if } a \in (0, \frac{1}{2}).
		\end{array}
		\right.
	\end{equation}
\end{proposition}

\begin{proof}
	When $a = 0$ the result is immediate from Remark \ref{rem:clifford_torus}.
	
	\textbf{Part 1:} We show that $x_a$ has exactly one zero in $[0, \frac{\pi}{2}]$.
	As $s$ increases from $0$ to $\frac{\pi}{2}$, $\phi(a, s)$ increases strictly from $0$ to $\phi(a, \frac{\pi}{2}) \in (\frac{\pi}{2}, \frac{\pi}{\sqrt{2}}]$, so $\cos\phi(a, s)$ decreases strictly from $1$ to $\cos\phi(a, \frac{\pi}{2})$, which is negative. Therefore, $\cos\phi(a, s)$ has exactly one zero in $[0, \frac{\pi}{2}]$ and the same is true of $x_a$.
	
	\textbf{Part 2:} Let $a > 0$.
	
	\textbf{Part 2.1:} We prove that if $s \in (0, \frac{\pi}{2})$ and $x_a(s) < 0$, then $\dot{x}_a(s) < 0$.
	To do this, compute $\dot{x}_a$ and express it as
	\begin{equation} \label{eq:derivative_first_coordinate_s3}
		\dot{x}_a(s) = \left(\frac{1}{2} - a\cos 2s \right) a \sin(2s) x_a(s) - y_a(s) \dot{\phi}(a, s).
	\end{equation}
	As in the proof of Proposition \ref{prop:one_parameter_family_embedded}, $\dot{\phi} > 0$ and $y_a(s) > 0$ when $s \in (0, \frac{\pi}{2})$, so the second term in \eqref{eq:derivative_first_coordinate_s3} is positive. Also, $\sin(2s) > 0$, so that if $x_a < 0$, then the first term in \eqref{eq:derivative_first_coordinate_s3} is negative, and hence $\dot{x}_a < 0$.
	
	\textbf{Part 2.2:} We show that $x_a(s_1(a)) < 0$.
	Multiplying $f_a$ by $(\frac{1}{2} - a\cos 2s)^\frac{1}{2}$ gives
	\begin{equation}
		\left(\frac{1}{2} - a\cos 2s \right)^\frac{1}{2} f_a(s) = a \sin(2s) y_a(s) + \left(\frac{1}{4} - a^2 \right)^\frac{1}{2} z_a(s) x_a(s),
	\end{equation}
	and since $f_a(s_1(a)) = 0$ by definition, we have
	\begin{equation} \label{eq:first_orthogonal_intersection_s3}
		a \sin(2s_1) \, y_a(s_1) + \left(\frac{1}{4} - a^2 \right)^\frac{1}{2} z_a(s_1) \, x_a(s_1) = 0.
	\end{equation}
	From the proof of Proposition \ref{prop:one_parameter_family_embedded}, $s_1(a) < \frac{\pi}{2}$. Then, as in part 2.1, the first term in \eqref{eq:first_orthogonal_intersection_s3} is positive, so the second term is negative, so $x_a(s_1(a)) < 0$.
	
	\textbf{Part 2.3:} We prove that $x_a(s) \ge x_a(s_1(a))$ for $s \in [-s_1(a), s_1(a)]$.
	$x_a$ is even, so it suffices to consider $s \in [0, s_1(a)]$. We know that $x_a(0) = (\frac{1}{2} - a)^\frac{1}{2} > 0$ and, by part 2.2, $x_a(s_1(a)) < 0$. So $x_a$ has a zero $s_* \in (0, s_1(a))$, and by part 1, this is the only zero of $x_a$ in $[0, \frac{\pi}{2}]$. Therefore $x_a \ge 0$ on $[0, s_*]$ and $x_a < 0$ on $(s_*, s_1(a)]$. This proves the claim for $s \in [0, s_*]$, and part 2.1 shows that $x_a$ is strictly decreasing on $(s_*, s_1(a)]$, proving the claim in this interval too. This concludes the case of positive $a$.
	
	\textbf{Part 3:} Let $a < 0$. Now we make the necessary changes to deal with negative $a$. Arguing as in part 2.1 shows that if $s \in (0, \frac{\pi}{2})$ and $x_a(s) > 0$, then $\dot{x}_a(s) < 0$. Then the argument of part 2.2 implies that $x_a(s_1(a)) > 0$, in contrast to the positive $a$ case. Finally, as in part 2.3, we consider $s \in [0, s_1(a)]$. Now both $x_a(0)$ and $x_a(s_1(a))$ are positive, so the only zero of $x_a$ must be $s_* \in (s_1(a), \frac{\pi}{2}]$. But then $x_a$ is positive in $[0, s_1(a)]$, so it is strictly decreasing, proving the result for negative $a$.
	
	\eqref{eq:radius_geodesic_balls} follows from part 2.3 and its counterpart for negative $a$.
\end{proof}

\begin{figure}
	\centering
	\includegraphics[
	width=0.7\textwidth
	]{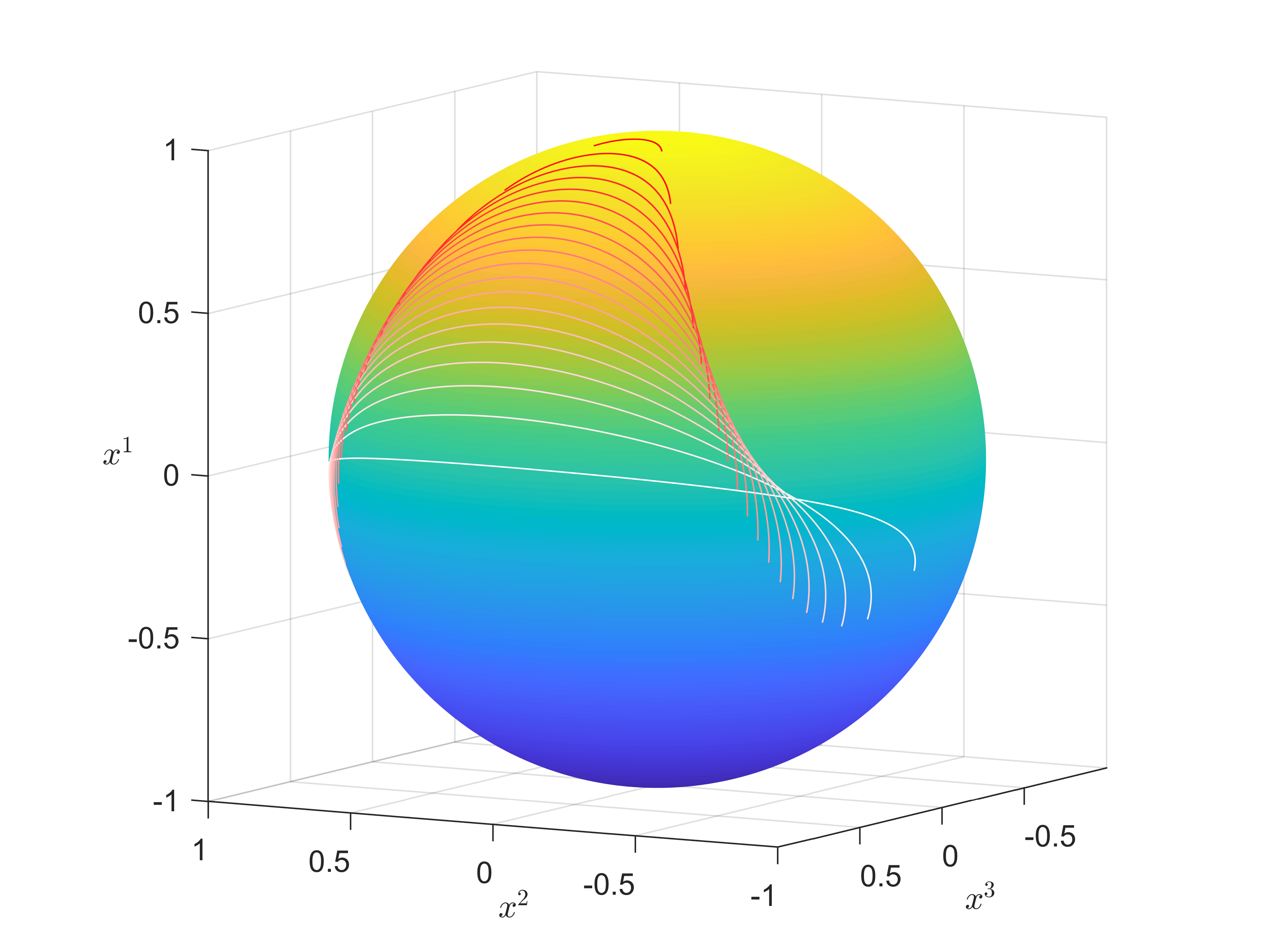}
	\caption{Generating curve segments $\gamma_a$ for twenty members of the one-parameter family of embedded free boundary minimal annuli described in Proposition \ref{prop:one_parameter_family_in_ball}, where $a$ varies linearly from $-0.49$ (red) to $0.49$ (white)}
	\label{fig:fbma_positive_negative_a}
\end{figure}

\begin{remark} \label{rem:li_xiong_converse}
	Let $\phi_0 = 0$. In \cite[Remark 9]{li_xiong_2018}, Li and Xiong suggest that for each $r \in (0, \frac{\pi}{2})$, there is some $a \in (-\frac{1}{2}, 0)$ and some $s_* > 0$ such that $X_a: [-s_*, s_*] \times \bb{S}^1 \to \bb{S}^3$ is a free boundary minimal annulus in $B_r(p_N) \subset \bb{S}^3$. This is then used by Lima and Menezes \cite{lima_menezes_2023} and Medvedev \cite{medvedev_2023}.
	
	Our Proposition \ref{prop:one_parameter_family_in_ball} implies that, for each $a \in (-\frac{1}{2}, 0)$, there is some $r(a) \in (0, \frac{\pi}{2})$ and some $s_1(a) > 0$ such that $X_a: [-s_1(a), s_1(a)] \times \bb{S}^1 \to \bb{S}^3$ is a free boundary minimal annulus in $B_{r(a)}(p_N) \subset \bb{S}^3$. Since these surfaces are known to have zero mean curvature from do Carmo and Dajczer \cite{doCarmo_dajczer_1983}, the key things that have to be proved are the existence of an orthogonal intersection with some geodesic sphere $\partial B_{r(a)}(p_N)$ \emph{and} the fact that the surface remains inside the corresponding geodesic ball $B_{r(a)}(p_N)$. While the former follows fairly easily, the latter is not obvious, as can be seen from the second plot of Figure \ref{fig:otsuki_2_3_phi0_0_fbma}. There we have an example of a free boundary minimal annulus with boundary in a geodesic sphere $\partial B_r(p_N)$ that is not contained in either of the two geodesic balls bounded by $\partial B_r(p_N)$.
	
	In particular, \cite[Theorem 4]{lima_menezes_2023} relies crucially on the fact that $X_a: [-s_1(a), s_1(a)] \times \bb{S}^1 \to \bb{S}^3$ maps into $B_{r(a)}(p_N)$, and therefore depends on our Proposition \ref{prop:one_parameter_family_in_ball}. Similarly, Medvedev's index calculation \cite[Theorem 1.3]{medvedev_2023} depends on \cite[Theorem 4]{lima_menezes_2023}, and hence also on Proposition \ref{prop:one_parameter_family_in_ball}.
	
	Finally, we note that to fully establish the claim in \cite[Remark 9]{li_xiong_2018}, we would have to show that $r: (-\frac{1}{2}, 0) \to (0, \frac{\pi}{2})$ given by $r(a) = \arccos x_a(s_1(a))$ is bijective (or at least surjective). We were not able to show this, despite there being strong numerical evidence that it is true.
\end{remark}

\section{Free boundary minimal annuli of revolution in Otsuki tori} \label{sec:otsuki}

We now turn to Otsuki tori and show that each Otsuki torus contains an abundance of immersed free boundary minimal annuli. In proving this result it will be important to consider the symmetries of various functions of a real variable. Given a function $u: \bb{R} \to \bb{R}$ and a point $x_0 \in \bb{R}$, we will say that $u$ is odd about $(x_0, u(x_0))$ if $\frac{u(2x_0 - x) + u(x)}{2} = u(x_0)$ for all $x$.

\begin{theorem} \label{thm:counting_fbma_in_otsuki_tori}
	Let $p, q \in \bb{N}$, $p, q$ coprime, with $\frac{1}{2} < \frac{p}{q} < \frac{1}{\sqrt{2}}$, and let $a \in (0, \frac{1}{2})$ be the unique solution to $C_a = \frac{2p\pi}{q}$.
	
	If $q$ is odd, then $\Sigma_a(0)$ contains at least $2p$ immersed free boundary minimal annuli with boundary in a geodesic sphere, and $\Sigma_a(\frac{\pi}{2})$ contains at least 2 isometric immersed free boundary minimal annuli with boundary in a geodesic sphere.
	
	If $q$ is even, then $\Sigma_a(0)$ contains at least $2p$ immersed free boundary minimal annuli with boundary in a geodesic sphere, isometric in pairs, and $\Sigma_a(\frac{\pi}{q})$ contains at least $2p$ immersed free boundary minimal annuli with boundary in a geodesic sphere, isometric in pairs.
\end{theorem}

\begin{proof}
	As in the proof of Lemma \ref{lem:immersed_torus_rational_pi}, $\phi(a, s + \pi) = \phi(a, s) + \frac{2p\pi}{q}$ for all $s$, so that $X_a$ and $f_a$ are $q\pi$-periodic in $s$. Therefore we consider $X_a$ and $f_a$ to be defined on $\bb{R}/q\pi\bb{Z}$ and will pick representative intervals appropriate to the situation.
	
	\textbf{Part 1:} We start by showing that $f_a$ has at least $2p$ zeros in $\bb{R}/q\pi\bb{Z} = [-\frac{q\pi}{2}, \frac{q\pi}{2})$, independently of the value of $\phi_0$.
	
	Since $C_a = \frac{2p\pi}{q}$, we have $\phi(a, \frac{k\pi}{2}) = \phi_0 + k\frac{p\pi}{q}$ for $k \in \bb{Z}$, as in the proof of Theorem \ref{thm:infinitely_many_fbma}, and then
	\begin{equation}
		f_a \left(\frac{k\pi}{2} \right) = \left(\frac{1}{4} - a^2 \right)^\frac{1}{2} \left(\frac{1}{2} + a (-1)^k \right)^\frac{1}{2} \cos \left(\phi_0 + k\frac{p\pi}{q} \right)
	\end{equation}
	for any $k \in \bb{Z}$. Note that if $k \in \bb{Z}$ is such that $\cos(\phi_0 + k\frac{p\pi}{q}) = 0$, then $f_a(\frac{k\pi}{2}) = 0$; and if $k \in \bb{Z}$ is such that $\cos(\phi_0 + k\frac{p\pi}{q}) \cdot \cos(\phi_0 + (k+1)\frac{p\pi}{q}) < 0$, then $f_a(\frac{k\pi}{2}) \cdot f_a(\frac{(k+1)\pi}{2}) < 0$, so $f_a$ has a zero in $(\frac{k\pi}{2}, \frac{(k+1)\pi}{2})$.
	
	Suppose $\cos(\phi_0 - p\pi) \neq 0$. We claim that the number of zeros plus the number of sign changes (i.e., consecutive entries whose product is negative) in the finite sequence
	\begin{equation} \label{eq:otsuki_cosine_sequence}
		\cos \left(\phi_0 + k\frac{p\pi}{q} \right), \quad k \in \{-q, -q+1, \dots, q\},
	\end{equation}
	is $2p$. To see this, note that the continuous function $\cos\theta$ has exactly $2p$ zeros in $[\phi_0 - p\pi, \phi_0 + p\pi]$, for its initial value is nonzero by hypothesis and the total change in the argument is $2p\pi$, that is, $p$ full rotations around the unit circle. These zeros occur at $\theta = (i + \frac{1}{2})\pi$ for $i \in \bb{Z}$ such that $(i + \frac{1}{2})\pi \in (\phi_0 - p\pi, \phi_0 + p\pi)$.
	
	But note that $\phi_0 + k\frac{p\pi}{q}$ is an arithmetic sequence that progresses in increments of $\frac{p\pi}{q} < \pi$, so there is an element of this sequence in every interval of length $\pi$ that intersects $[\phi_0 - p\pi, \phi_0 + p\pi]$. Hence, for each integer $i$ with $(i + \frac{1}{2})\pi \in (\phi_0 - p\pi, \phi_0 + p\pi)$, either $(i + \frac{1}{2})\pi = \phi_0 + k\frac{p\pi}{q}$ for some $k \in \{-q, -q+1, \dots, q\}$, or else there is some $k \in \{-q, -q+1, \dots, q\}$ such that $\phi_0 + k\frac{p\pi}{q} \in ((i-\frac{1}{2})\pi, (i+\frac{1}{2})\pi)$ and $\phi_0 + (k+1)\frac{p\pi}{q} \in ((i+\frac{1}{2})\pi, (i+\frac{3}{2})\pi)$. That is, each of the $2p$ zeros of $\cos\theta$ in $(\phi_0 - p\pi, \phi_0 + p\pi)$ either gives rise to a zero of \eqref{eq:otsuki_cosine_sequence} or gives rise to a sign change in \eqref{eq:otsuki_cosine_sequence}, but not both. This proves the claim and shows that $f_a$ has at least $2p$ zeros in $[-\frac{q\pi}{2}, \frac{q\pi}{2})$.
	
	If $\cos(\phi_0 - p\pi) = 0$, then also $\cos(\phi_0 + p\pi) = 0$, so repeating the argument above gives $2p + 1$ zeros or sign changes of \eqref{eq:otsuki_cosine_sequence}. These correspond to at least $2p + 1$ zeros of $f_a$, except that the last such zero happens at $s = \frac{q\pi}{2}$, which is outside the representative interval we are considering. So again we are left with at least $2p$ zeros of $f_a$ in $[-\frac{q\pi}{2}, \frac{q\pi}{2})$.
	
	\textbf{Part 2:} Suppose $q$ is odd.
	
	If $\phi_0 = 0$, then we have seen in the proof of Theorem \ref{thm:infinitely_many_fbma} that $\phi(a, \cdot)$ is odd and hence $f_a$ and $x_a$ are both even. Since, in this case, $f_a(-\frac{q\pi}{2})$ and $f_a(0)$ are nonzero, part 1 guarantees we have at least $p$ zeros of $f_a$ in $(0, \frac{q\pi}{2})$. Denoting these positive zeros of $f_a$ by $s_1 < s_2 < \dots < s_p$, by Lemma \ref{lem:characterization_orthogonal_intersections} we may parametrize $2p$ immersed free boundary minimal annuli by $X_a:[-s_i, s_i] \times \bb{S}^1 \to \bb{S}^3$ and $X_a:[s_i, q\pi - s_i] \times \bb{S}^1 \to \bb{S}^3$ for each $i = 1, 2, \dots, p$. Note that we made use of two different representative intervals for $\bb{R}/q\pi\bb{Z}$ in the previous line, $[-\frac{q\pi}{2}, \frac{q\pi}{2})$ and $[0, q\pi)$. For each such $i$, the two immersed minimal annuli share their boundary and that boundary lies on the geodesic sphere $\bb{S}^3 \cap \{x^1 = x_a(s_i)\}$. See Figure \ref{fig:otsuki_2_3_phi0_0_fbma}.

	\begin{figure}
		\centering
		\includegraphics[width=\textwidth]{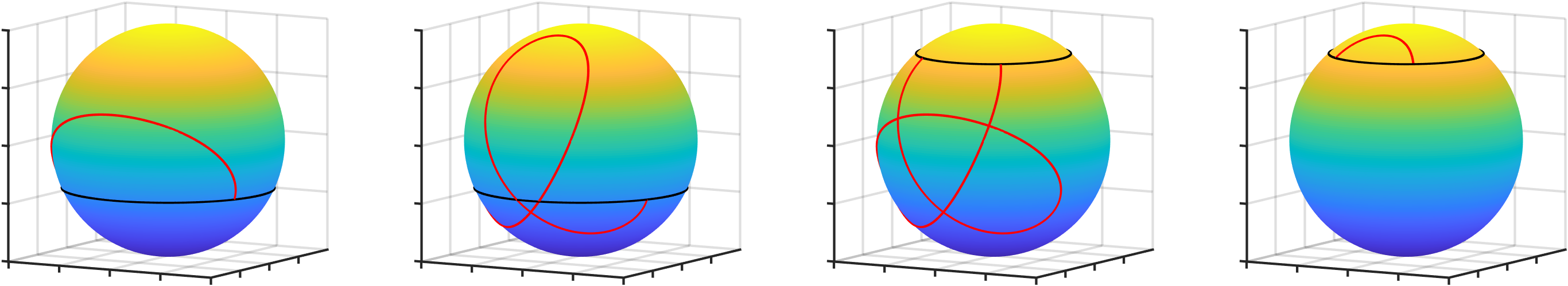}
		\caption{Four segments of the generating curve $\gamma_a$ (red) for the Otsuki torus $\Sigma_a(0)$ corresponding to $\frac{p}{q} = \frac{2}{3}$, each generating one of the immersed free boundary minimal annuli described in Theorem \ref{thm:counting_fbma_in_otsuki_tori}, together with the geodesic spheres $\partial B_r^2(p_N)$ (black) each segment intersects orthogonally}
		\label{fig:otsuki_2_3_phi0_0_fbma}
	\end{figure}
	
	If $\phi_0 = \frac{\pi}{2}$, then $f_a(-\frac{q\pi}{2}) = f_a(0) = 0$ and $x_a(-\frac{q\pi}{2}) = x_a(0) = 0$, so by Lemma \ref{lem:characterization_orthogonal_intersections}, $X_a: [-\frac{q\pi}{2}, 0] \times \bb{S}^1 \to \bb{S}^3$ and $X_a: [0, \frac{q\pi}{2}] \times \bb{S}^1 \to \bb{S}^3$ are 2 immersed free boundary minimal annuli with the same boundary, and that boundary lies on the equator $\bb{S}^3 \cap \{x^1 = 0\}$. Also we see that $\phi(a, -s) = \pi - \phi(a, s)$, which implies that $x_a$ is odd whereas $y_a$ and $z_a$ are even. That is, the point $X_a(-s, \theta)$ is a reflection of $X_a(s, \theta)$ across $x^1 = 0$, and thus the two immersed free boundary minimal annuli in this case are isometric. See Figure \ref{fig:otsuki_2_3_phi0_pi_2_fbma}.
	
	\begin{figure}
		\centering
		\includegraphics[width=\textwidth]{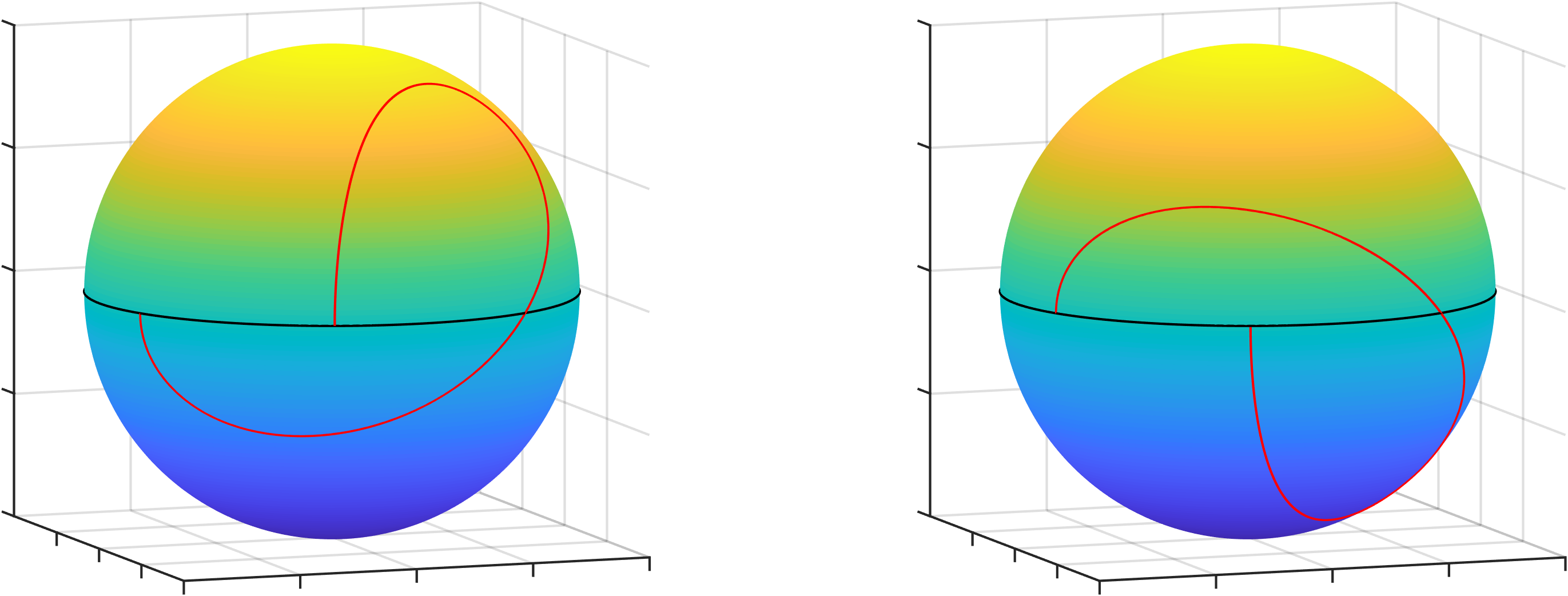}
		\caption{Two segments of the generating curve $\gamma_a$ (red) for the Otsuki torus $\Sigma_a(\frac{\pi}{2})$ corresponding to $\frac{p}{q} = \frac{2}{3}$, each generating one of the immersed free boundary minimal annuli described in Theorem \ref{thm:counting_fbma_in_otsuki_tori}, together with the geodesic sphere $\bb{S}^2 \cap \{x^1 = 0\}$ (black) both segments intersect orthogonally}
		\label{fig:otsuki_2_3_phi0_pi_2_fbma}
	\end{figure}
	
	\textbf{Part 3:} Suppose $q$ is even.
	
	\textbf{Part 3.1:} Let $\phi_0 = 0$.
	
	In this case, as in the beginning of part 2, $f_a$ and $x_a$ are both even, so the zeros of $f_a$ are symmetric about the origin. It turns out that these zeros are also symmetric about $s = \frac{q\pi}{4}$, which we now show.
	
	It is immediate that $\phi_0 = 0$ implies that $\phi(a, \frac{q\pi}{2} - s) = -\phi(a, s) + p\pi$. This, together with the fact that $q$ is now even, implies that $f_a(\frac{q\pi}{2} - s) = -f_a(s)$, that is, $f_a$ is odd about $(\frac{q\pi}{4}, 0)$. In particular, $f_a(\frac{q\pi}{4}) = 0$ and the zeros of $f_a$ in $(0, \frac{q\pi}{2})$ are symmetric about $s = \frac{q\pi}{4}$.
	
	Since, by part 1, we have at least $2p$ zeros of $f_a$ in $[-\frac{q\pi}{2}, \frac{q\pi}{2})$, $f_a(-\frac{q\pi}{2})$ and $f_a(0)$ are nonzero, and $f_a$ is even, we must have at least $p$ zeros of $f_a$ in $(0, \frac{q\pi}{2})$. Then, $f_a(\frac{q\pi}{4}) = 0$ and $f_a$ being odd about $(\frac{q\pi}{4}, 0)$ imply that we have at least $\frac{p-1}{2}$ zeros in $(0, \frac{q\pi}{4})$, which we denote by $s_1 < s_2 < \dots < s_{\frac{p-1}{2}}$. Now let $s_{\frac{p+1}{2}} = \frac{q\pi}{4}$ and $s_{p+1-i} = \frac{q\pi}{2} - s_i$ for $i = 1, \dots, \frac{p-1}{2}$. This gives $p$ zeros of $f_a$ in $(0, \frac{q\pi}{2})$, denoted $s_1 < \dots < s_p$.
	
	As in part 2, Lemma \ref{lem:characterization_orthogonal_intersections} now gives $2p$ immersed free boundary minimal annuli, parametrized by $X_a:[-s_i, s_i] \times \bb{S}^1 \to \bb{S}^3$ and $X_a:[s_i, q\pi - s_i] \times \bb{S}^1 \to \bb{S}^3$, with boundary in the geodesic sphere $\bb{S}^3 \cap \{x^1 = x_a(s_i)\}$, for $i = 1, 2, \dots, p$. We show these are isometric in pairs by proving that $X_a:[-s_i, s_i] \times \bb{S}^1 \to \bb{S}^3$ is a rotation of $X_a:[s_{p+1-i}, q\pi - s_{p+1-i}] \times \bb{S}^1 \to \bb{S}^3$ for $i = 1, 2, \dots, p$.
	
	To see this, note that $p$ must now be odd and $\phi(a, \frac{q\pi}{2} + s) = \phi(a, s) + p\pi$. Letting $R(\beta)$ denote a rotation by $\beta$ radians on the $x^1x^2$-plane, we compute directly that $R(\pi) X_a(\frac{q\pi}{2} + s, \theta) = X_a(s, \theta)$, so that $X_a:[-s_i, s_i] \times \bb{S}^1 \to \bb{S}^3$ is a rotation of $X_a:[\frac{q\pi}{2} - s_i, \frac{q\pi}{2} + s_i] \times \bb{S}^1 \to \bb{S}^3$. But, from the way we defined $s_i$, we have $[\frac{q\pi}{2} - s_i, \frac{q\pi}{2} + s_i] = [s_{p+1-i}, q\pi - s_{p+1-i}]$, proving the claim of the previous paragraph. See Figure \ref{fig:otsuki_5_8_phi0_0_fbma}.

	\begin{figure}
		\centering
		\includegraphics[width=\textwidth]{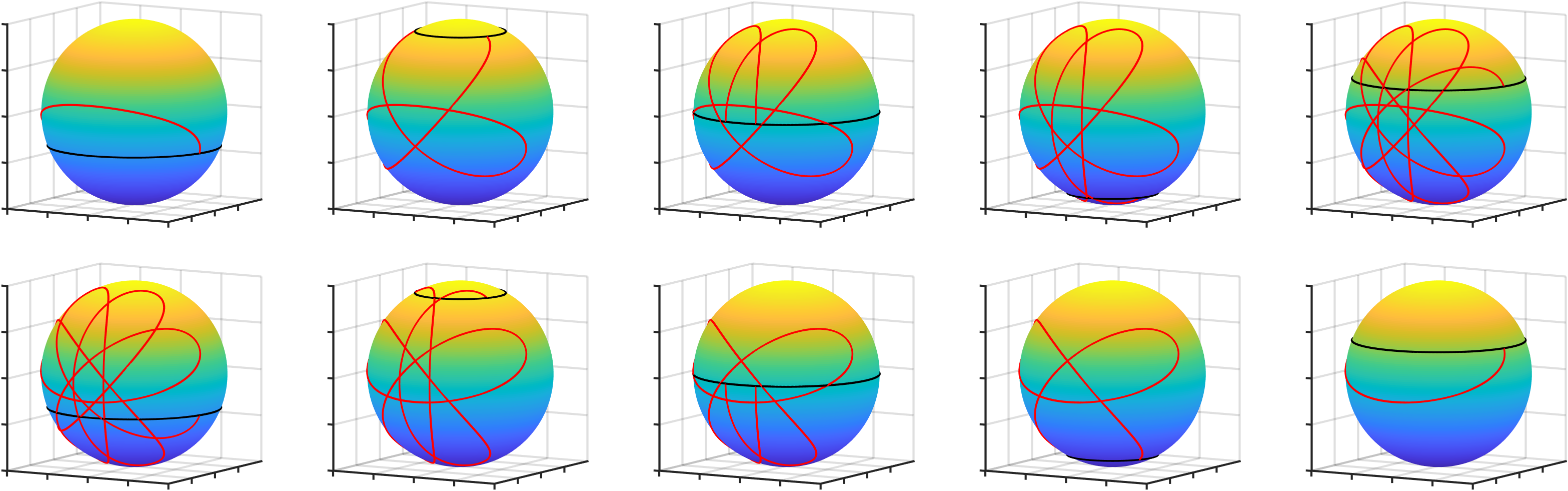}
		\caption{Ten segments of the generating curve $\gamma_a$ (red) for the Otsuki torus $\Sigma_a(0)$ corresponding to $\frac{p}{q} = \frac{5}{8}$, each generating one of the immersed free boundary minimal annuli described in Theorem \ref{thm:counting_fbma_in_otsuki_tori}, together with the geodesic spheres $\partial B_r^2(p_N)$ (black) each segment intersects orthogonally}
		\label{fig:otsuki_5_8_phi0_0_fbma}
	\end{figure}
	
	\textbf{Part 3.2:} Let $\phi_0 = \frac{\pi}{q}$.
	
	As before, part 1 still implies we have at least $2p$ zeros of $f_a$ in $[-\frac{q\pi}{2}, \frac{q\pi}{2})$. However, $f_a$ is in general no longer even, so its zeros are not in general symmetric about the origin and it is not obvious which pairs of zeros, if any, correspond to points on the same geodesic sphere. To overcome this we will show that, despite the different parametrizations, $\Sigma_a(\frac{\pi}{q}) = \Sigma_{-a}(0)$ as sets, and the result will then follow from part 3.1.
	
	We change notation slightly by writing $\gamma_a(\phi_0, s)$ instead of $\gamma_a(s)$ for the generating curve of $\Sigma_a(\phi_0)$. We compute directly that $\gamma_a(\phi_0, s + \pi) = R(\frac{2p\pi}{q}) \gamma_a(\phi_0, s)$ for all $s$, so letting $s$ range over all of $\bb{R}$ we find $\gamma_a(\phi_0, \bb{R}) = R(\frac{2p\pi}{q}) \gamma_a(\phi_0, \bb{R})$. Repeating this $n$ times gives $\gamma_a(\phi_0, \bb{R}) = R(\frac{2np\pi}{q}) \gamma_a(\phi_0, \bb{R})$ for every $n \in \bb{N}$. Now take $n$ to be a solution of $np \equiv 1 \pmod{q}$, and it follows that $\gamma_a(\phi_0, \bb{R}) = R(\frac{2\pi}{q}) \gamma_a(\phi_0, \bb{R})$.
	
	Now we relate $\gamma_a(\phi_0, \bb{R})$ with $\gamma_{-a}(0, \bb{R})$. Computing from the definition,
	\begin{equation} \label{eq:gamma_a_positive_negative_a}
		\gamma_{-a}\left(0, s - \frac{n\pi}{2} \right) = R \left(\phi_0 + \frac{np\pi}{q} \right) \gamma_a(\phi_0, s)
	\end{equation}
	for all $s$ and $n \in \bb{N}$. Using $\phi_0 = \frac{\pi}{q}$ and taking $n$ to be a solution of $np \equiv 1 \pmod{2q}$ yields $\gamma_{-a}(0, s - \frac{n\pi}{2}) = R(\frac{2\pi}{q}) \gamma_a(\frac{\pi}{q}, s)$, so that $\gamma_{-a}(0, \bb{R}) = R(\frac{2\pi}{q}) \gamma_a(\frac{\pi}{q}, \bb{R})$. This, together with the conclusion of the previous paragraph, implies that $\gamma_{-a}(0, \bb{R}) = \gamma_a(\frac{\pi}{q}, \bb{R})$, showing that $\Sigma_{-a}(0) = \Sigma_a(\frac{\pi}{q})$.
	
	We may now apply the argument of part 3.1 verbatim to $\Sigma_{-a}(0)$ as it did not rely on the sign of $a$, concluding the proof. See Figure \ref{fig:otsuki_5_8_phi0_pi_q_fbma}.
	
	\begin{figure}
		\centering
		\includegraphics[width=\textwidth]{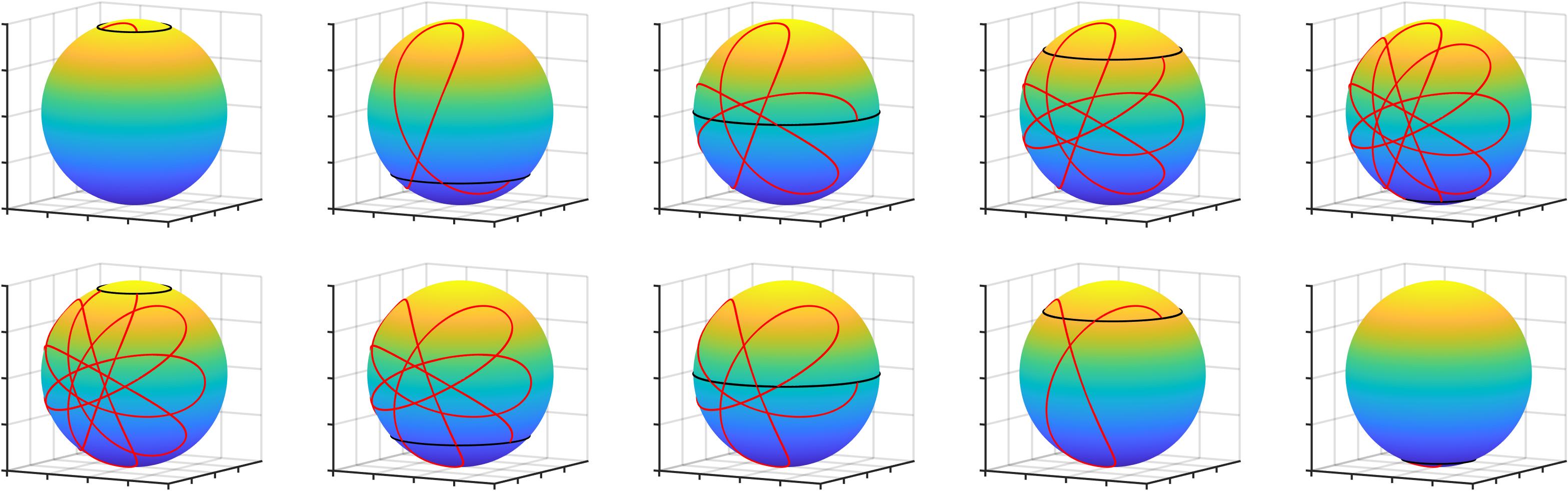}
		\caption{Ten segments of the generating curve $\gamma_a$ (red) for the Otsuki torus $\Sigma_a(\frac{\pi}{8})$ corresponding to $\frac{p}{q} = \frac{5}{8}$, each generating one of the immersed free boundary minimal annuli described in Theorem \ref{thm:counting_fbma_in_otsuki_tori}, together with the geodesic spheres $\partial B_r^2(p_N)$ (black) each segment intersects orthogonally}
		\label{fig:otsuki_5_8_phi0_pi_q_fbma}
	\end{figure}	
\end{proof}

\section{Properties} \label{sec:properties}

The large number of explicit free boundary minimal surfaces we described above opens the door to the investigation of many related questions. While some of these questions have been addressed for free boundary minimal submanifolds in geodesic balls of $\bb{S}^n$ smaller than a hemisphere, here we mention two that generalize easily to larger geodesic balls.

\subsection{Volume and boundary volume}

A very interesting set of isoperimetric inequalities was recently established by Lee and Seo \cite[Theorem 2.7]{lee_seo_2023} for free boundary minimal submanifolds in geodesic balls $B_r(p_N) \subset \bb{S}^n$ with $r < \frac{\pi}{2}$. By following their proof closely one sees that the hypothesis $r < \frac{\pi}{2}$ is not actually necessary, establishing the next proposition.
\begin{proposition} \label{prop:vol_bdry_vol}
	Let $r \in (0, \pi)$ and $\Sigma^k$ be a compact free boundary minimal submanifold of $B_r(p_N) \subset \bb{S}^n$. Then
	\begin{equation} \label{eq:ratio_boundary_volume_over_volume_sphere}
		k \cot(r) \le \frac{|\partial\Sigma|}{|\Sigma|} \le k \left( \frac{1 + \cos(r)}{2\sin(r)} \right).
	\end{equation}
\end{proposition}

Of course, the lower bound in \eqref{eq:ratio_boundary_volume_over_volume_sphere} is vacuous when $r \ge \frac{\pi}{2}$ as it becomes nonpositive in this range, but the upper bound is interesting for all $r \in (0, \pi)$.

\subsection{Index}

Lima and Menezes \cite[Lemma 1]{lima_menezes_2023} recently proved a lower bound on the Morse index of free boundary minimal hypersurfaces $\Sigma$ in a geodesic ball $B_r(p_N) \subset \bb{S}^n$ with $r < \frac{\pi}{2}$. Their proof relies on the fact that, when $r < \frac{\pi}{2}$, the first eigenfunctions of the Dirichlet-to-Neumann map for the Helmholtz equation on $\Sigma$ are the constants. While that is no longer true when $r \ge \frac{\pi}{2}$, we show that there is a good replacement result for this range.

To do that we need a couple of results from our unpublished report \cite{oliveira_2021}. The first is an eigenvalue characterization of free boundary minimal submanifolds in a geodesic ball $B_r(p_N) \subset \bb{S}^n$, where $r \in (0, \pi)$.
\begin{proposition} \label{prop:fbms_eigenvalue_char_sphere}
	Let $r \in (0, \pi)$ and $\Sigma^k$ be a submanifold of $B_r(p_N) \subset \bb{S}^n$ with nonempty boundary such that $\partial\Sigma \subset \partial B_r(p_N)$. Let $\eta$ be the outward unit conormal to $\Sigma$ along $\partial\Sigma$. Then $\Sigma$ is a free boundary minimal submanifold of $B_r(p_N)$ if and only if
	\begin{equation} \label{eq:bvp_x1_fbms_sphere}
		\left\{
		\begin{array}{ll}
			\Delta_\Sigma x^1 + k x^1 = 0 & \text{in } \Sigma,\\ [5pt]
			\partial_\eta x^1 = -\sin(r) & \text{on } \partial\Sigma,
		\end{array}
		\right.
	\end{equation}
	and
	\begin{equation}
		\left\{
		\begin{array}{ll}
			\Delta_\Sigma x^i + k x^i = 0 & \text{in } \Sigma,\\ [5pt]
			\partial_\eta x^i = \cot(r) x^i & \text{on } \partial\Sigma,
		\end{array}
		\right.
	\end{equation}
	for $i = 2, \dots, n+1$.
\end{proposition}

The only difference between the above and \cite[Proposition 1]{lima_menezes_2023} is the form of the boundary condition in \eqref{eq:bvp_x1_fbms_sphere}, which here is valid across all $r \in (0, \pi)$. This allows us to give a unified proof of the following balancing result (see \cite[Proposition 11]{oliveira_2021}), without distinguishing cases.
\begin{lemma} \label{lem:balancing}
	Let $r \in (0, \pi)$ and $\Sigma^k$ be a compact free boundary minimal submanifold of $B_r(p_N) \subset \bb{S}^n$. Then
	\begin{equation}
		\int_{\partial\Sigma} x^i = 0
	\end{equation}
	for $i = 2, \dots, n+1$.
\end{lemma}

\begin{proof}
	For any $i = 2,\dots,n+1$,
	\begin{equation}
		\begin{aligned}
			\cos(r) \cot(r) \int_{\partial\Sigma} x^i &= \int_{\partial\Sigma} x^1 \partial_\eta x^i\\
			&= \int_{\partial\Sigma} x^i \partial_\eta x^1 + \int_\Sigma \left( x^1 \Delta_\Sigma x^i - x^i \Delta_\Sigma x^1 \right)\\
			&= -\sin(r) \int_{\partial\Sigma} x^i + \int_\Sigma \left( x^1 (-k x^i) - x^i (-k x^1) \right)\\
			&= -\sin(r) \int_{\partial\Sigma} x^i,
		\end{aligned}
	\end{equation}
	where we have used Green's formulas for the Laplacian and Proposition \ref{prop:fbms_eigenvalue_char_sphere}. The result follows from the fact that $\cos(r) \cot(r) + \sin(r) = \csc(r)$ is never zero for $r \in (0,\pi)$.
\end{proof}

We may now use the previous lemma to extend the index bound of Lima and Menezes \cite[Lemma 1]{lima_menezes_2023} to free boundary minimal hypersurfaces in a geodesic ball $B_r(p_N) \subset \bb{S}^n$ with $r \ge \frac{\pi}{2}$.
\begin{proposition} \label{prop:index_fbm_hypersurfaces}
	Let $\Sigma^{n-1}$ be a compact two-sided free boundary minimal hypersurface of $B_r(p_N) \subset \bb{S}^n$ that is not contained in a hyperplane through the origin. If $r \in (\frac{\pi}{2}, \pi)$, then the Morse index of $\Sigma$ is at least $n$. If $r = \frac{\pi}{2}$, then the Morse index of $\Sigma$ is at least $n+1$.
\end{proposition}

\begin{proof}
	In \cite[Lemma 1]{lima_menezes_2023} it is shown that, when $r < \frac{\pi}{2}$, the index form evaluated on linear functions reduces to
	\begin{equation} \label{eq:index_form_general}
		S(\varphi_v, \varphi_v) = -\int_\Sigma |\II|^2 \varphi_v^2 \, d\mu_\Sigma - (v^1)^2 \cot(r) |\partial\Sigma|,
	\end{equation}
	where $v \in \bb{R}^{n+1}$ and $\varphi_v(x) = \langle x, v \rangle \in C^\infty(\Sigma)$. There, the authors use the fact that $\varphi = 1$ is a first eigenfunction of the Dirichlet-to-Neumann map for the Helmholtz equation to conclude that $\int_{\partial\Sigma} x^i = 0$ for $i = 2, \dots, n+1$. While it is no longer true that $\varphi = 1$ is a first eigenfunction when $r \ge \frac{\pi}{2}$, the conclusion still holds due to Lemma \ref{lem:balancing}, and therefore \eqref{eq:index_form_general} is in fact valid for any $r \in (0, \pi)$.
	
	Hence, when $r \in (\frac{\pi}{2}, \pi)$, the index form is negative definite on $\{ \varphi_v \colon v \in \bb{R}^{n+1}, v^1 = 0 \}$, an $n$-dimensional space. When $r = \frac{\pi}{2}$, \eqref{eq:index_form_general} becomes $S(\varphi_v, \varphi_v) = -\int_\Sigma |\II|^2 \varphi_v^2 \, d\mu_\Sigma$, which is negative definite on $\{ \varphi_v \colon v \in \bb{R}^{n+1} \}$, an $(n+1)$-dimensional space.
\end{proof}

Shortly after the work of Lima and Menezes \cite{lima_menezes_2023}, Medvedev \cite{medvedev_2023} studied index questions further and in particular showed that the free boundary minimal surfaces of revolution we have described that are contained in a geodesic ball $B_r(p_N) \subset \bb{S}^3$ with $r < \frac{\pi}{2}$ have index 4. It would be interesting to see how much of their work generalizes to larger geodesic balls.

\begin{remark} \label{rem:eigenvalue_char_doesnt_need_ball}
	It should be noted that Proposition \ref{prop:fbms_eigenvalue_char_sphere}, Lemma \ref{lem:balancing} and Proposition \ref{prop:index_fbm_hypersurfaces} apply more generally to free boundary minimal submanifolds $\Sigma \subset \bb{S}^n$ with boundary in a geodesic sphere $\partial B_r(p_N)$, even if $\Sigma$ is not contained in a ball, as long as $\Sigma$ hits the geodesic sphere from the same side along all its boundary components. This means the results of this subsection apply to all the surfaces of Figures \ref{fig:fbma_general_a}, \ref{fig:otsuki_2_3_phi0_0_fbma}, \ref{fig:otsuki_5_8_phi0_0_fbma} and \ref{fig:otsuki_5_8_phi0_pi_q_fbma}, where we may have to apply them to a geodesic ball centered at the south pole instead, but not to the surfaces of Figure \ref{fig:otsuki_2_3_phi0_pi_2_fbma}.
\end{remark}

\noindent \textbf{Acknowledgements.} I thank my supervisor Ailana Fraser for many discussions and for her suggestions for improvement to multiple versions of this manuscript.

\bibliographystyle{amsplain}
\bibliography{references}

\end{document}